\documentclass[11pt]{amsart}

\usepackage{amsmath}
\usepackage{amsfonts}

\usepackage{amscd}
\usepackage{amsthm}
\usepackage{amssymb} \usepackage{latexsym}

\usepackage{eufrak}
\usepackage{euscript}
\usepackage{epsfig}

\usepackage{graphics}
\usepackage{array}
\usepackage{enumerate}
\usepackage{dsfont}
\usepackage{color}
\usepackage{wasysym}

\usepackage{hyperref}
\usepackage{pdfsync}

\newcommand{\asinh}{{\mbox{arcsinh}}}
\theoremstyle{theorem}
\newtheorem{thm}{Theorem}[section]
\theoremstyle{corollary}
\newtheorem{coro}{Corollary}[section]
\theoremstyle{lemma}
\newtheorem{lemma}{Lemma}[section]
\theoremstyle{definition}
\newtheorem{defi}{Definition}[section]
\theoremstyle{remark}
\newtheorem{rem}{Remark}[section]

\begin{document}

\title{Davies-Gaffney-Grigor'yan Lemma on Simplicial complexes}
\author{Bobo Hua}

\address{Bobo Hua,School of Mathematical Sciences, LMNS, Fudan University, Shanghai 200433, China; Shanghai Center for Mathematical Science, Fudan University, Shanghai 200433, China. }
\email{bobohua@fudan.edu.cn}

\author{Xin Luo}
\address{Xin Luo, College of Mathematics and Econometrics, Hunan University,Changsha 410082,China}
\email{xinluo@hnu.edu.cn}

\begin{abstract}
We prove Davies-Gaffney-Grigor'yan lemma for heat kernels of bounded discrete Hodge Laplacians on
simplicial complexes.

%\noindent{Keywords:} Discrete Hodge Laplacian, Simplicial complex, Heat semigroup, Heat kernel
\end{abstract}
\maketitle

\section{introduction}
The Davies-Gaffney-Grigor'yan Lemma, denoted as DGG Lemma in short below, is useful for heat kernel estimates on both manifolds and graphs. The DGG lemma on manifolds can be described as follows

\begin{lemma}[Davies-Gaffney-Grigor'yan]\label{DGGManifold}
Let $M$ be a complete Riemannian manifold and $p_t(x,y)$ the minimal heat kernel on $M$. For any two measurable  subsets $B_1$ and $B_2$ of $M$ and $t>0,$ we have
\begin{equation}\label{e:DGG Riemannian}\int_{B_1}\int_{B_2} p_t(x,y)d\mathrm{vol}(x)d\mathrm{vol}(y) \leq \sqrt{\mathrm{vol}(B_1)\mathrm{vol}(B_2)}\exp\left(-\mu t\right)\exp\left(- \frac{d^2(B_1,B_2)}{4t}\right),\end{equation}
where $\mu$ is the greatest lower bound of the $\ell^2$-spectrum of the Laplacian on $M$ and $d(B_1,B_2)=\inf_{x_1\in B_1,x_2\in B_2}d(x_1,x_2)$ the distance between $B_1$ and $B_2$.
\end{lemma}
Davies firstly proved a lemma of this type in \cite{7} by adpoting the argument of Gaffney \cite{14}. Li and Yau also proved an earlier version of this lemma in \cite{23}. Later Grigor'yan proved the lemma in \cite{15} and introduced the term $\exp(-\mu t)$ on the right hand side which gives the sharp speed of decay of
the heat kernel as $t\rightarrow \infty$ when $\mu > 0.$

Recently, variants of the Davies-Gaffney-Grigor'yan Lemma were proved by Bauer, Hua and Yau \cite{2,3} on graphs.

Weighted graphs are defined to be the set of vertices $V$ and edges $E=\{(x,y)|x, y \in V\}$ with
a measure function $m:V\in x\mapsto m_{x}\in (0,\infty)$ and an edge weight function $\mu:E\mapsto\mu_{xy}\in[0,\infty).$ The measure of any vertices subset $A$ is defined as the
sum of measures of vertices, i.e. $m(A)=\sum_{x\in A}m_{x}$. Intrinsic metrics on graphs were introduced by Frank, Lenz and Wingert in \cite{13} and a pseudo metric $\rho$ is called intrinsic if
$\sum_{y\in V}\mu_{xy}\rho^{2}(x,y)\leq m_{x}.$ The quantity, $s:=sup\{\rho(x,y)|x,y\in V, \mu_{xy}>0\},$ where the supremum is taken over all pairs $(x,y)$ with $\mu_{xy}>0$ is the jump size of an intrinsic metric $\rho.$

According to \cite{2,3},
\begin{thm}
Let $(V,\mu,m)$ be a weighted graph with an intrinsic metric $\rho$ with finite jump size $s>0.$ Let $A,B$ be two subsets in $V$ and $f,g\in \ell^2_m$ with $\mathrm{supp} f\subset A, \mathrm{supp} g\subset B,$ then
$$
|\langle e^{t\Delta}f,g\rangle|\leq e^{-\lambda t-\zeta_s(t,\rho(A,B))}\|f\|_{\ell^2_m}\|g\|_{\ell^2_m}
$$
where $\lambda$ is the bottom of the $\ell^2$-spectrum of Laplacian and
$$\zeta_s(t,r)=\frac{1}{s^2}(rs\asinh{\frac{rs}{t}}-\sqrt{t^2+r^2s^2}+t). \quad t>0,r\geq0$$
Moreover,
$$\sum_{y\in B}\sum_{x\in A}p_t(x,y)m_xm_y\leq \sqrt{m(A)m(B)}e^{-\lambda t-\zeta_s(t,\rho(A,B))}$$
where $p_t(x,y)$ is the minimal heat kernel of the graph.
\end{thm}
The DGG lemmas can be used to obtain heat kernel estimates and eigenvalue estimates. For instance,
in conjunction with the Li-Yau inequality, pointwise heat kernel estimates were obtained for mainfolds in \cite{22,23}. And Chung, Grigor'yan and Yau \cite{5} obtained eigenvalue estimates with DGG lemma on compact Riemannian mainfolds. Moreover, combining the Harnack inequality and the DGG lemma, \cite{2} proved the heat kernel estimates on graphs satisfying the exponential curvature dimension inequality and eigenvalue estimates on finite graphs.

In this paper, we prove the DGG Lemma  for discrete Hodge Laplacian on simplicial complexes. Discrete Hodge Laplacians on simplicial complexes are generlizations of Hodge Laplacian on differential forms which encode the information of topology and geometry. Horak and Jost \cite{18,19} obtained properties of spectra of (discrete) Hodge Laplacians on special simplicial complexes and effects of constructions on spectrum. In this paper, we consider bounded Hodge Laplacians and related heat kernel to prove the DGG lemma. For bounded Hodge Laplacian $\mathcal{L}_i,$ it defines the heat semigroup $\{e^{-t\mathcal{L}_i}\}_{t\geq 0}$ which is a family of bounded self-adjoint operators. Then $e^{-t\mathcal{L}_i}$ has a heat kernel. We provide one kind of appropriate weight metrics on simplicial complexes which behave like the intrinsic metrics on graphs and with this kind of  metrics we are able to prove DGG Lemma for the continuous time heat kernel related to bounded Hodge Laplacian on simplicial complexes adopting the methods in \cite{3,6}.

We state the DGG lemma on simplicial complexes as follows, for the precise definitions of the quantities used we refer to Section 2.
\begin{thm}\label{sharps}
Let $(K,w)$ be an oriented weighted simplicial complex with an intrinsic metric $\rho$ with jump size $s>0.$ Let $A,B$ be two subsets in $S_{i}(K)$ and $f,g\in \ell^2_w$ acting on $i$-simplices with $\mathrm{supp} f\subset A, \mathrm{supp} g\subset B,$ then
$$|\langle e^{-t\mathcal{L}_i}f,g\rangle|\leq e^{-\lambda t-\zeta_s(t,\rho(A,B))}\|f\|_{\ell^2_w}\|g\|_{\ell^2_w},$$
In particular,
\begin{equation}
|\sum_{F'\in B}\sum_{F\in A}p_t(F,F')w(F)w(F')|\leq \sqrt{w(A)w(B)}e^{-\lambda t-\zeta_s(t,\rho(A,B))}
\end{equation}
where $\lambda$ is the bottom of the $\ell^2$-spectrum of bounded Hodge Laplacian $\mathcal{L}_i$ and $\zeta_s(t,\rho)=\frac{1}{s^2}(\rho s\asinh{\frac{\rho s}{t}}-\sqrt{t^2+\rho^2s^2}+t).$
\end{thm}
Moreover, by setting $A=\{F\}, B=\{F'\},$ we obtain an upper bound for heat kernel between
two $i$-simplicial faces for $i=\{0,1,...,dim K\}.$
$$|p_t(F,F')|\leq \frac{1}{\sqrt{w(F)w(F')}}e^{-\lambda t-\zeta_s(t,\rho(F,F'))}$$
\begin{rem}
Our result is restricted to bounded Hodge Laplacians. It will be interesting to know whether it holds for unbounded cases.
\end{rem}

\section{Simplicial complexes}
In this section we introduce the setting and definitions used throughout this paper. An abstract simplicial complex $K$ on vertices
set $V$ is a collection of subsets of $V$ which is closed under incluson, i.e. if $F \in K$ and
$F'\subset F$, then $F'\in K$. An $i$-face of $K$ is an element of cardinality $i+1$, and we denote the set of all $i$-faces by $S_i(K)$. The dimension of simplicial complex $K$ is the largest dimension of the faces it contains. For instance, graphs are $1$-dim simplicial complexes and the dimension of empty simplex is $(-1).$ For two $(i+1)$-faces with one common $i$-face, we say that they are $i$-down neighbors and two $i$-faces sharing an $(i+1)$-face are called $(i+1)$-up neighbors.

We say that a face $F$ is oriented if the ordering of its vertices is given and write as $[F].$ Two orderings of its vertices are said to determine the same or opposite orientation when the related permutation is even or odd permutation.

A simplicial complex $K$ is called weighted if there is a weight function $w$ on the set of all faces satisfying
\begin{equation*}
w: \bigcup^{dimK}_{i=-1}S_i(K)\rightarrow (0,+\infty).
\end{equation*}
The weight of a face $F$ is $w(F)$ and the weight of a subset $A$ is the sum of weights of faces contained in $A,$ i.e.
$w(A)= \sum_{F \in A} w(F).$

Considering oriented simplicial complexes, the $i$-th chain group with coefficients in $\mathbb{R}$ is a vector space over $\mathbb{R}$ with basis $B_i(K,\mathbb{R})=\{[F]|F \in S_i(K)\}$. The cochain group is
the dual of the chain group with basis given by the set of functions $\{e_{[F]}|[F]\in B_i(K,\mathbb{R})\}$ where
\begin{equation}
e_{[F]}([F'])=\left\{ \begin{array}{r@{\quad}l}
 1 &if [F']=[F]\\
0 &  otherwise.
\end{array}\right.
\end{equation}
We denote the $i$-th chain group and cochain group by $C_i(K,\mathbb{R})$ and $C^{i}(K,\mathbb{R})$ respectively.
The coboundary operator $\delta_i$ $:$ $C^{i}(K,\mathbb{R})\rightarrow C^{i+1}(K,\mathbb{R})$ is the linear map
$$(\delta_if)(v_0,...,v_{i+1})=\sum_{k=0}^{i+1}(-1)^{k}f(v_0,...,\widehat{v_k},...,v_{i+1}),$$
where $\widehat{v_k}$ means that the vertex $v_k$ has been omitted. With the weight function $w$, inner product on cochain group is given as follows
$$(f,g)_{C^{i}}=\sum_{F\in S_i(K)}f([F])g([F])w(F).$$
And the adjoint
$\delta_i^{\ast}$:
$C^{i+1}(K,\mathbb{R})\rightarrow C^{i}(K,\mathbb{R})$ of the coboundary map $\delta_i$
is defined such that
$$
(\delta_if,g)_{C^{i+1}}=(f,\delta_i^{\ast}g)_{C^{i}},
$$
for every $f \in C^{i}(K,\mathbb{R})$ and $g \in C^{i+1}(K,\mathbb{R})$.
Eckmann generalized the graph Laplacian to simplicial complexes and proved the discrete version of
the Hodge theorem \cite{10} which can be formulated as
$$
ker(\delta_i^{\ast}\delta_i+\delta_{i-1}\delta_{i-1}^{\ast})\cong \tilde{H^{i}}(K,\mathbb{R})
$$
for the sequence of linear transformations
$$
...\overset{\delta_{i+1}}\leftarrow C^{i+1}(K,\mathbb{R})\overset{\delta_{i}}\leftarrow C^{i}(K,\mathbb{R})
\overset{\delta_{i-1}}\leftarrow...\leftarrow C^{-1}(K,\mathbb{R})\leftarrow 0
$$
where$$\mathcal{L}_i(K)=\delta_i^{\ast}\delta_i+\delta_{i-1}\delta_{i-1}^{\ast}$$
is the $i$-dimensional Hodge Laplacian.

The operators $\delta_i^{\ast}\delta_i$ and $\delta_{i-1}\delta_{i-1}^{\ast}$ are called the $i$-up and $i$-down
Hodge Laplace operator and are denoted by $\mathcal{L}^{up}_i(K)$ and $\mathcal{L}^{down}_i(K)$ respectively. Moreover, the operators $\mathcal{L}^{up}_i(K)$, $\mathcal{L}^{down}_i(K)$ and  $\mathcal{L}_i(K)$ are self-adjoint and non-negative.

Let $\overline{F}=\{v_0,...,v_{i+1}\}$ be an $(i+1)$-face of a simplicial complex $K$ and
$F=\{v_0,...,\widehat{v_k},...,v_{i+1}\}$ be an $i$-face of  $\overline{F}$. Then the \emph{boundary of the oriented face}
$[\overline{F}]$ is
\begin{equation*}
\partial[\overline{F}]=\sum_{k}(-1)^{k}[v_0,...,\widehat{v_k},...,v_{i+1}],
\end{equation*}
and the sign of $[F]$ in the boundary of $[\overline{F}]$ is denoted by the $sgn([F],\partial[\overline{F}])$
which is equal to $(-1)^{k}$. If $[F]$ is not the boundary of $[\overline{F}]$, we say $sgn([F],\partial[\overline{F}])$ is equal to $0$. And for the sake of brevity, we will write $sgn([F],\partial[\overline{F}])$ as
$\sigma_{F\overline{F}}$ below.

According to \cite{18,19}, the
$i$-up and $i$-down Hodge Laplace operators are given by
\begin{equation}\label{eq1}
(\mathcal{L}^{up}_{i})f([F])
=\sum_{\substack{\overline{F}\in S_{i+1}(K):\\
F\in \partial\overline{F}}}\frac{w(\overline{F})}{w(F)}f([F])
+\sum_{\substack{F'\in S_i:F\neq F',\\F,F'\in \partial\overline{F}}}\frac{w(\overline{F})}{w(F)}\sigma_{F\overline{F}}
\sigma_{F'\overline{F}}f([F'])
\end{equation}
and
\begin{equation}\label{eq2}
(\mathcal{L}^{down}_{i})f([F])
=\sum_{E \in \partial F}\frac{w(F)}{w(E)}f([F])
+\sum_{F':F\cap F'=E}\frac{w(F')}{w(E)}\sigma_{EF}\sigma_{EF'}f([F'])
\end{equation}
where $\partial\overline{F}$ is the set of all $i$-faces of $\overline{F}$ by abuse of notation.

For the degree of an $i$-face $F$, it is the sum of the weights of all faces that contain $F$
in its boundary, that is
$$
\mathrm{deg} F=\sum_{\substack{\overline{F}\in S_{i+1}(K):\\F\in \partial\overline{F}}}w(\overline{F}).
$$

Moreover, if the weight function $w$ on $K$ satisfies
$$
w(F)=\mathrm{deg} F
$$
for every $F \in S_i(K)$ which is not a facet of $K,$
the Hodge Laplace operator is called the \emph{weighted normalized Hodge Laplacian operator}.
If a simplicial complex satisfies that there is a positive integer $M$ such that
$$
\sharp\{F \in S_i(K)| E \in \partial F\}\leq M < \infty
$$
holds for all $E \in S_{i-1}(K),$
then the Hodge Laplacian operator $\mathcal{L}_{i}$ acting on functions on $i$-simplicial faces of oriented simplicial complexes $(K,w)$ is bounded from $l_{w}^{2}$ to $l_{w}^{2}$  if and only if
\begin{equation}\label{bounded}
b:=\sup_{F \in S_j(K)}\frac{1}{w(F)} \sum_{F \in \partial\overline{F}} w(\overline{F})<\infty \quad for\quad  j=i,i-1\tag{$\ast$}
\end{equation}

Next, we will introduce the following notations for different $i$-faces $F$ and $F'.$

$(1)\tau_{FEF'}:=\frac{w(F)w(F')}{w(E)}.$

$(2)w^{up}_{FF'}:=\left\{ \begin{array}{r@{\quad}l}
w(\overline{F}), &if F',F\in \partial\overline{F};\\
0,&  otherwise.
\end{array}\right.$

$(3)w^{down}_{FF'}:=\left\{ \begin{array}{r@{\quad}l}
\tau_{FEF'}, &if F'\cap F=E;\\
0, &  otherwise.
\end{array}\right.$

$(4)w_{FF'}:=w^{up}_{FF'}+w^{down}_{FF'}.$

$(5)\mathrm{Deg}(F):=\frac{\mathrm{deg F}}{w(F)}.$

For convenience, we extend the notations to the total set $S_i(K)\times S_i(K),$
such that $\tau_{FEF}=w^{up}_{FF}=w^{down}_{FF}=w_{FF}=0.$
And we write $F\sim F'$ if $w_{FF'}\neq 0$ holds. Also for simplicity, we will omit $``[~]" $ for oriented faces in the following.
\begin{lemma}[Green's formula]\label{l:Green}Let $\mathcal{L}_{i}$ be the bounded Hodge Laplacian operator. Then for all $f,g\in \ell_{w}^{2}$ acting on $i$-simplices
$$
(\mathcal{L}_{i} f,g)=(\delta_i f,\delta_i g)+(\delta_{i-1}^{\ast}f,\delta_{i-1}^{\ast}g)
$$
In particular,
\begin{eqnarray*}
\sum_{F\in S_i(K)}(\mathcal{L}_{i} f)(F)g(F)w(F)
=\sum_{F\in S_i(K)}f(F)g(F)(\sum_{F\in \partial \overline{F}} w(\overline{F})+\sum_{E\in \partial F}\tau_{FEF})\\
 +\sum_{F\neq F'}f(F')g(F)(w(\overline{F})\sigma_{F\overline{F}}\sigma_{F'\overline{F}}+\tau_{FEF'}\sigma_{EF}\sigma_{EF'})
\end{eqnarray*}

\end{lemma}
\begin{proof}
By direct computation
\begin{eqnarray*}
(\mathcal{L}_{i} f,g)
&=&(\delta_if,\delta_ig)+(\delta_{i-1}^{\ast}f,\delta_{i-1}^{\ast}g)\\
&=&\sum_{\overline{F}\in S_{i+1}(K)}\sum_{F,F'\in \partial\overline{F}}\sigma_{F\overline{F}}\sigma_{F'\overline{F}}f(F)g(F')w(\overline{F})\\
&+&\sum_{E \in S_{i-1}(K)}\sum_{F,F'\in S_{i}(K)}\sigma_{EF}\sigma_{EF'}\tau_{FEF'}f(F)g(F')\\
&=&\sum_{F\in S_i(K)}f(F)g(F)(\sum_{F\in \partial \overline{F}} w(\overline{F})+\sum_{E\in \partial F}\tau_{FEF})\\
&+&\sum_{F\neq F'}f(F')g(F)(w(\overline{F})\sigma_{F\overline{F}}\sigma_{F'\overline{F}}+\tau_{FEF'}\sigma_{EF}\sigma_{EF'})
\end{eqnarray*}
\end{proof}
In this subsection, we introduce a kind of metrics on simplicial complexes which could be viewed as generalizations of the intrinsic metrics on graphs introduced in \cite{13}. Indeed, intrinsic metrics on graphs have been applied successfully to various problems, see \cite{1,4,11,12,16,20,21}.
\begin{defi}[Intrinsic metric]
A pseudo metric $\rho$ is called an intrinsic metric with respect to a simplicial complex $(K,\omega)$ if for
all $F \in S_{i}(K)$
\begin{equation}
\sum_{F'\in S_i(K)}w_{FF'}\rho^{2}(F,F')\leq w(F).
\end{equation}
\end{defi}
\begin{rem}
In our setting, there always exists an intrinsic metric on a weighted simplicial complex, see Definition \ref{defi23}
mimicking the definition introduced by Huang \cite{17}. In general, intrinsic metrics are not unique.
\end{rem}
\begin{defi}\label{defi23}
We define a function $\mu(F,F')$ by
\begin{align*}
\mu(F,F')=
&\min\left\{\sqrt{\frac{w(F)}{\sum\limits_{F'}w_{FF'}}},\sqrt{\frac{w(F')}{\sum\limits_{F''}w_{F'F''}}},
1\right\}
\end{align*}
for all pairs $(F,F')$ satisfying $w_{FF'}\neq 0$. It naturally induces a metric for different $i$-dim simplicial faces of $(K,\omega)$ by
$$
\rho(F,F'):=\inf\left\{\sum\mu(F_{j},F_{j+1}):F_0=F,...,F_j,...,F_m=F', \ s.t.\ \forall 0\leq j \leq m, F_{j}\sim F_{j+1} \right\}.
$$
where the infimum is taken over all such chains $F_0=F,...,F_j,...,F_m=F'$ of $i$-faces between $F$ and $F'$ with $w_{F_jF_{j+1}}\neq 0$ . If there is no such chain
between $F$ and $F'$, we define $\rho(F,F')=\infty$ and $\rho(F,F)=0$ for the same $i$-dim simplicial face.
\end{defi}
It is obvious that the above metrics do satisfy the condition $(6)$.
\begin{rem}
The intrinsic metric on graphs in \cite{17} is defined as follows:
$$
\delta(x,y)=\inf_{x=x_0\sim...\sim x_n=y}\sum^{n-1}_{i=0}(\mathrm{Deg}(x_i)\vee \mathrm{Deg}(x_{i+1})\vee 1)^{-\frac{1}{2}},
x,y \in X
$$
When the intrinsic metric defined in (\ref{defi23}) applied to graphs,
there is
$$\rho(x,y)=\delta(x,y)$$
\end{rem}
For bounded Hodge Laplacian satisfying (\ref{bounded}), there is a kind of canonical intrinsic metric analogous to the combinational distance on graphs.
\begin{defi}\label{defi24}
We can define another form of intrinsic metric between
$F, F' \in S_i(K)$ by
\begin{equation*}
\rho(F,F')=\left\{ \begin{array}{r@{\quad}l}\frac{\inf\sum\mu(F_j,F_{j+1})}{(i+1)\sqrt{b}},& if F\neq F'; \\
0, & if F=F'.
\end{array}\right.
\end{equation*}
where the infimum is taken over all such chains $F_0=F,...,F_j,...,F_m=F'$ of $i$-faces between $F$ and $F'$ with $w_{F_jF_{j+1}}\neq 0.$ If there is no such chain
between $F$ and $F'$, we define $\rho(F,F')=\infty$. And
\begin{equation*}
\mu(F_j,F_{j+1})=\left\{ \begin{array}{r@{\quad}l}1,& if w_{F_jF_{j+1}}\neq 0; \\
0, & otherwise.
\end{array}\right.
\end{equation*}
\end{defi}
We will show the above metric is also an intrinsic metric.
\begin{proof} It is easy to see that
$$\rho(F,F')\leq \frac{1}{(i+1)\sqrt{b}},\quad w_{FF'}\neq 0$$
such that
\begin{eqnarray*}&&
\sum_{F'\in S_i}w_{FF'}\rho^{2}(F,F')\nonumber\\
&\leq& \frac{1}{b(i+1)^{2}}\left((i+1)\sum_{F\in\partial\overline{F}}w(\overline{F})+
\sum_{E\in \partial F}\sum_{F'\neq F}\tau_{FEF'}\right)\nonumber\\
&\leq&\frac{1}{b(i+1)^{2}}\left((i+1)bw(F)+w(F)
\sum_{E\in \partial F}\frac{bw(E)-w(F)}{w(E)}\right)\nonumber\\
&\leq&w(F)
\end{eqnarray*}
\end{proof}
\begin{rem}
For graph case, the normalized Laplacian must satisfy (\ref{bounded}) while
for higher dimensional simplicial complexes,
the normalized Hodge Laplacian operator
may not satisfies (\ref{bounded}). So for normalized Hodge Laplacian on higher dimensional simplicial complexes, the metric defined in Definition 2.3. may be not suitable.
\end{rem}

The \emph{jumps size} $s$ of a pseudo metric $\rho$ is given by
\begin{align*}
    s:=\sup\{\rho(F,F')\mid F,F'\in S_i(K), F\sim F'\}\in[0,\infty].
\end{align*}
If there is no $F$ and $F'$ satisfying $w_{FF'}\neq 0$, it
is reasonable to define $s=\infty$. From now on, $\rho$ always denotes an intrinsic metric and $s$ denotes its jump size.
A function $f:S_i(K)\to\mathbb{R} $ is Lipschitz (w.r.t.the metric $\rho$) if $|f(F)-f(F')|\leq \kappa\rho(F,F')$ for any $F,F'\in S_i(K).$ The minimal constant $\kappa$ that satisfies the above inequality is called the Lipschitz constant of $f$ and $supp f$ means the maximal set of simplicial faces $F$ satisfying $f(F)\neq 0,$ i.e.
$supp f =\{F|F \in S_i(K),f(F)\neq 0\}.$
\section{Proof of main theorem}
Let $(K,w)$ be an oriented weighted simplicial complex.
\begin{defi}
We say $u:[0,\infty)\times S_i(K)\to \mathbb{R}$ solves heat equation if
\begin{equation}
\begin{cases}
\frac{\partial}{\partial t}u(t,F) = -\mathcal{L}_{i} u(t,F),\\
u(0,F) = f(F).
\end{cases}
\end{equation}
The heat kernel, $p_t(F,F'),$ is defined as the solution with the initial condition $f(F)=\frac{1}{w(F')}\delta_{F'}(F).$
For a general initial data $f(F),$ the solution can be written as $$u(t,F) = \sum_{F'\in S_i(K)} p_t(F,F') f(F') w(F').$$
\end{defi}
The integral maximum principle on Riemannian manifolds was introduced by Grigor'yan \cite{15}. We prove a variant of the integral maximum principle on simplicial complexes.
\begin{lemma}\label{l:monotonicity Delmotte}
Let $(K,w)$ be an oriented weighted simplicial complex with an  intrinsic metric $\rho$ with jump size $s>0,$ and $\zeta$ be a bounded Lipschitz function on $i$-simplices with the Lipschitz constant $\kappa.$ Let $f:[0,\infty)\times S_i(K)\to\mathbb{R}$ solve the heat equation on $K$ and $\lambda$ be the first eigenvalue of the bounded Hodge Laplacian. Set
$$E(t):= \sum_{F\in
S_i(K)}f^2(t,F)e^{\zeta(F)}w(F).$$ Then $$\exp\left({2\lambda t-\frac{2}{s^2}(\cosh(\frac{\kappa s}{2})-1)t}\right)E(t)$$ is  nonincreasing in $t\in [0,\infty)$.
\end{lemma}
\begin{proof}
From dominated convergence theorem,
$$E'(t)=2\sum_{F \in S_i{(K)}}f(t,F)\partial_{t} f(t,F)e^{\zeta(F)}w(F)$$
Since $f$ solves the heat equation and together with Green's formula, we obtain
\begin{eqnarray*}
 E'(t)&=&-2\sum_{F\in S_i(K)}f^{2}(F)e^{\zeta(F)}\sum_{F \in \partial\overline{F}}w(\overline{F})
   -2\sum_{F\in S_i(K)}f^{2}(F)e^{\zeta(F)}\sum_{E\in \partial F}\tau_{FEF} \\
   &-&2\sum_{F\neq F'}f(F)f(F')e^{\zeta(F)}\sigma_{F\overline{F}}\sigma_{F'\overline{F}}w(\overline{F})\\
    &-&2\sum_{F\neq F'}f(F)f(F')e^{\zeta(F)}\sigma_{EF}\sigma_{EF'}\tau_{FEF'}\\
   &=&-2\sum_{F\in S_i(K)}f^{2}(F)e^{\zeta(F)}\sum_{F \in \partial\overline{F}}w(\overline{F})
   -2\sum_{F\in S_i(K)}f^{2}(F)e^{\zeta(F)}\sum_{E\in \partial F}\tau_{FEF} \\
   &-&\sum_{F\neq F'}f(F)f(F')\sigma_{F\overline{F}}\sigma_{F'\overline{F}}w(\overline{F})(e^{\zeta(F)}+e^{\zeta(F')})\\
   &-&\sum_{F\neq F'}f(F)f(F')\sigma_{EF}\sigma_{EF'}\tau_{FEF'}(e^{\zeta(F)}+e
   ^{\zeta(F')})\\
\end{eqnarray*}
Moreover
\begin{eqnarray*}
E'(t) &=&-2(\mathcal{L}_if e^{\frac{1}{2}\zeta},f e^{\frac{1}{2}\zeta})
    +2(\mathcal{L}_if e^{\frac{1}{2}\zeta},f e^{\frac{1}{2}\zeta})+E'(t) \ \ \ \ \ \ \ \ \ \ \ \ \ \ \ \ \ \ \ \ \ \ \  \\
    &=&I+II
\end{eqnarray*}
where $I=-2(\mathcal{L}_if e^{\frac{1}{2}\zeta},f e^{\frac{1}{2}\zeta}),$
$II=2(\mathcal{L}_if e^{\frac{1}{2}\zeta},f e^{\frac{1}{2}\zeta})+E'(t)$

For the first term, by the Rayleigh quotient of the first eigenvalue,\ \ \ \ \ \ \ \ \ \ \ \ \ \ \ \ \ \ \ \ \ \ \ \ \ \ \
$$I\leq -2\lambda\sum f^2e^{\zeta} w.$$
For the second term, using the Green's formula again
\begin{eqnarray*}
II&=&\sum_{F\in S_i(K)}2f^{2}(F)e^{\zeta(F)}(\sum_{F\in \partial \overline{F}} w(\overline{F})+\sum_{E\in \partial F}\tau_{FEF})\\
&+&2\sum_{F\neq F'}f(F')f(F)e^{\frac{1}{2}\zeta(F)+\frac{1}{2}\zeta(F')}(w(\overline{F})\sigma_{F\overline{F}}\sigma_{F'\overline{F}}+\tau_{FEF'}\sigma_{EF}\sigma_{EF'})\\
&-&2\sum_{F\in S_i(K)}f^{2}(F)e^{\zeta(F)}\sum_{F \in \partial\overline{F}}w(\overline{F})
-2\sum_{F\in S_i(K)}f^{2}(F)e^{\zeta(F)}\sum_{E\in \partial F}\tau_{FEF} \\
&-&\sum_{F\neq F'}f(F)f(F')\sigma_{F\overline{F}}\sigma_{F'\overline{F}}w(\overline{F})(e^{\zeta(F)}+e^{\zeta(F')})\\
&-&\sum_{F\neq F'}f(F)f(F')\sigma_{EF}\sigma_{EF'}\tau_{FEF'}(e^{\zeta(F)}+e
^{\zeta(F')})\\
&=&-\sum_{F\neq F'}f(F)f(F')\sigma_{F\overline{F}}\sigma_{F'\overline{F}}w(\overline{F})
(e^{\frac{1}{2}\zeta(F)}-e^{\frac{1}{2}\zeta(F')})^{2}\nonumber\\
 &-&\sum_{F\neq F'}f(F)f(F')\sigma_{EF}\sigma_{EF'}\tau_{FEF'}
(e^{\frac{1}{2}\zeta(F)}-e^{\frac{1}{2}\zeta(F')})^{2}\nonumber\\
&=&-\sum_{F\neq F'}f(F)f(F')\sigma_{F\overline{F}}\sigma_{F'\overline{F}}w(\overline{F})
e^{\frac{1}{2}\zeta(F)+\frac{1}{2}\zeta(F')}
(e^{\frac{1}{4}\zeta(F)-\frac{1}{4}\zeta(F')}-e^{\frac{1}{4}\zeta(F')-\frac{1}{4}\zeta(F)})^{2}\nonumber\\
&-&\sum_{F\neq F'}f(F)f(F')\sigma_{EF}\sigma_{EF'}\tau_{FEF'}
e^{\frac{1}{2}\zeta(F)+\frac{1}{2}\zeta(F')}
(e^{\frac{1}{4}\zeta(F)-\frac{1}{4}\zeta(F')}-e^{\frac{1}{4}\zeta(F')-\frac{1}{4}\zeta(F)})^{2}\nonumber\\
\end{eqnarray*}
Because $2(\cosh\frac{x-y}{2}-1)=(e^{\frac{1}{4}x-\frac{1}{4}y}-e^{\frac{1}{4}y-\frac{1}{4}x})^{2}$, then
\begin{eqnarray*}
II&=&-2
   \sum_{F\neq F'}f(F)f(F')\sigma_{F\overline{F}}\sigma_{F'\overline{F}}w(\overline{F})
   e^{\frac{1}{2}\zeta(F)+\frac{1}{2}\zeta(F')}\left(\cosh\frac{\zeta(F)-\zeta(F')}{2}-1\right)\nonumber\\
   &-&2
   \sum_{F\neq F'}f(F)f(F')\sigma_{EF}\sigma_{EF'}\tau_{FEF'}
   e^{\frac{1}{2}\zeta(F)+\frac{1}{2}\zeta(F')}\left(\cosh\frac{\zeta(F)-\zeta(F')}{2}-1\right)\nonumber\\
\end{eqnarray*}
\begin{eqnarray*}
   &\leq&
   \sum_{F\neq F'}w(\overline{F})(f^{2}(F) e^{\zeta(F)}+f^{2}(F') e^{\zeta(F')})
   \left(\cosh\frac{\zeta(F)-\zeta(F')}{2}-1\right)\nonumber\\
  &+&
   \sum_{F\neq F'}\tau_{FEF'}(f^{2}(F) e^{\zeta(F)}+f^{2}(F') e^{\zeta(F')})
   \left(\cosh\frac{\zeta(F)-\zeta(F')}{2}-1\right)\nonumber\\
\end{eqnarray*}
From the symmetry of equation in $F$ and $F'$,
\begin{eqnarray*}
II&\leq&
   2\sum_{F\neq F'}w(\overline{F})f^{2}(F) e^{\zeta(F)}
   \left(\cosh\frac{\zeta(F)-\zeta(F')}{2}-1\right)\nonumber\\
   &+&
   2\sum_{F\neq F'}\tau_{FEF'}f^{2}(F) e^{\zeta(F)}
   \left(\cosh\frac{\zeta(F)-\zeta(F')}{2}-1\right)\nonumber\\
   &=&2\sum_{F\neq F'}f^{2}(F) e^{\zeta(F)}
   \left(\cosh\frac{\zeta(F)-\zeta(F')}{2}-1\right)(w^{up}_{FF'}+w^{down}_{FF'})\nonumber\\
   &=&2\sum_{F\neq F'}f^{2}(F) e^{\zeta(F)}
   \left(\cosh\frac{\zeta(F)-\zeta(F')}{2}-1\right)w_{FF'}\nonumber\\
\end{eqnarray*}
We claim that for any $F$ and $F'$ with $w_{FF'}\neq 0$,
$$\cosh\frac{\zeta(F)-\zeta(F')}{2}-1\leq \rho^{2}(F,F')\frac{1}{s^2}(\cosh\frac{\kappa s}{2}-1).$$ It suffices to consider $F\sim F'$ with $\rho(F,F')>0.$  Since $\zeta$ is a Lipschitz function with Lipschitz constant $\kappa,$
\begin{eqnarray*}
    &&\cosh\frac{\zeta(F)-\zeta(F')}{2}-1\\
    &\leq&\cosh\frac{\kappa \rho(F,F')}{2}-1
    =\rho^{2}(F,F')\frac{\cosh\frac{\kappa \rho(F,F')}{2}-1}{\rho^{2}(F,F')}\\
    &\leq&\rho^{2}(F,F')\frac{1}{s^2}(\cosh\frac{\kappa s}{2}-1),
\end{eqnarray*} where we used the monotonicity of the function
  $$t\mapsto \frac{1}{t^2}(\cosh\frac{\kappa t}{2}-1),\quad t>0,$$ and the definition of jump size $s$ of the metric $\rho.$
  This proves the claim.

  Hence by this claim
  \begin{eqnarray*}II&\leq& 2\sum_{F\neq F'}w_{FF'}f^2(t,F)e^{\zeta(F)}\rho^{2}(F,F')\frac{1}{s^2}(\cosh\frac{\kappa s}{2}-1)\\
  &\leq &\frac{2}{s^2}(\cosh\frac{\kappa s}{2}-1)\sum_{F}f^2(t,F)e^{\zeta(F)}w(F),
  \end{eqnarray*}
  where we used that $\rho$ is an intrinsic metric. Combining the estimates for the terms $I$ and $II,$ for any $t\geq 0,$
  $$E'(t)\leq (-2\lambda+\frac{2}{s^2}(\cosh\frac{\kappa s}{2}-1))E(t).$$
  which shows $$\exp\left({2\lambda t-\frac{2}{s^2}(\cosh(\frac{\kappa s}{2})-1)t}\right)E(t)$$
  is nonincreasing in $t\in [0,\infty)$.
\end{proof}

Given $s>0,$ for any fixed $t>0$ and $r\geq 0$, we denote
\begin{equation}\label{e:def eta}
\zeta_s(t,r)=-\inf_{\kappa>0}\left(\frac{1}{s^2}(\cosh\frac{\kappa s}{2}-1)t-\frac{\kappa}{2} r\right).
\end{equation}
It is easy to see that the infinium is attained at $\kappa_0=\frac{2}{s}\asinh{\frac{rs}{t}}$ and \begin{equation}\label{eq:eta}\zeta_s(t,r)=\frac{1}{s^2}\left(rs\asinh{\frac{rs}{t}}-\sqrt{t^2+r^2s^2}+t\right).\end{equation} Note that $\zeta_s(t,r)=\zeta(\frac{t}{s^{2}}, \frac{r}{s}),$ where $\zeta(t,r)=r\asinh{\frac{r }{t}}-\sqrt{t^2+r^2}+t.$ The function $\zeta(t,r)$ was already used by Davies, Pang, Delmotte to obtain eatimates of heat kernel  \cite{8,9,24}.

Moreover, according to \cite{9}
\begin{equation}
\begin{cases}
\zeta(t,r) \leq \frac{r^{2}}{2t},\text{for }t\geq0\\
 \zeta(t,r)\geq h \arcsin(h^{-1})\frac{r^{2}}{2t}, \text{for }t\geq hr.
\end{cases}
\end{equation}

Now, we are ready to prove the DGG lemma on simplicial complexes.
\begin{proof}
Denote $r=\rho(A,B).$ For any $\kappa>0,$ define $\zeta(F)=\kappa\rho(F,A)\wedge (\kappa\rho(A,B)+1),$ $F\in S_i(K).$ Then $\zeta$ is a Lipschitz function with Lipschitz constant at most $\kappa$ and for any function $g$ on $i$-simplices $$\sum_{B}|g|^2e^{\zeta}w\geq e^{\kappa r}\sum_{B}|g|^2w.$$ For $f\in\ell^2_w,$ let $f(t,F)=e^{-t\mathcal{L}_i}f(F).$ Then the above inequality for $g(\cdot)=f(t,\cdot)$ and Lemma \ref{l:monotonicity Delmotte} yield
\begin{eqnarray*}\sum_{F\in B}|f(t,F)|^2w(F)&\leq& e^{-\kappa r}E(t)\leq \exp\left(-2\lambda t+\frac{2}{s^2}(\cosh\frac{\kappa s}{2}-1)t-\kappa r\right)E(0)\\
&=&\exp\left(2(-\lambda t+\frac{1}{s^2}(\cosh\frac{\kappa s}{2}-1)t-\frac{\kappa}{2} r)\right)\sum_A f^2w,\end{eqnarray*} where we used that $\mathrm{supp} f\subset A.$  For fixed $s,t>0$ and $r\geq0,$ choose suitable $\kappa$ such that the following function $$\mathbb{R}^+\ni\kappa\mapsto {\frac{1}{s^2}(\cosh\frac{\kappa s}{2}-1)t-\frac{\kappa}{2} r}$$ attains the minimum. Then by \eqref{e:def eta} and \eqref{eq:eta} we have
$$\sum_{B}|f(t,F)|^2w
\leq e^{2(-\lambda t-\zeta_s(t,r))}\sum_A |f|^2w.$$
That is, for all $f\in \ell^2(A,w),$
$$\sup_{\substack{g\in \ell^2(B,w)\\\|g\|_{\ell^2_w}=1}}|\langle e^{-t\mathcal{L}_i}f,g\rangle|^2=\sum_B|e^{-t\mathcal{L}_i}f|^2w\leq e^{2(-\lambda t-\zeta_s(t,r))}\|f\|_{\ell^2_w}^2.$$ This proves the theorem.
\end{proof}
Using the properties of $\zeta,$ $(10),$ we obtain the following corollary.
\begin{coro}
Let $p_t(F,F')$ be the minimal heat kernel of the simplicial $K$ and $h > 0.$ Then there exists a constant $C(h,s)$ such that for any two subsets $A,B \subset S_{i}(K),t\geq sh\rho(A,B),$
\begin{equation}
|\sum_{F'\in B}\sum_{F\in A}p_t(F,F')w(F)w(F')|\leq \sqrt{w(A)w(B)}\exp\left(-\lambda t\right)\exp\left(-C\frac{\rho^{2}(A,B)}{4t}\right)
\end{equation}
\end{coro}
Noting that for the bounded Hodge Laplacian with $b < \infty,$ from Definition 2.3, the corresponding jump size $s$ is equal to $\frac{1}{(i+1)\sqrt{b}}.$ And we have the following corollary.
\begin{coro}
For a simplicial complex $K$ with the quantity $b< \infty$, the continuous
heat kernel satifies:
$$|p_t(F,F')|\leq  \frac{1}{\sqrt{w(F)w(F')}}e^{-\lambda t-\zeta_{\frac{1}{(i+1)\sqrt{b}}}(t,\rho(F,F'))}.$$
where $\rho$ is the distance defined in Definition (\ref{defi24}).
\end{coro}
When applied to graphs with the normalized Laplacian, the above corollary implies the Davies's heat kernel estimate, [3,$~$Corollary$~$1.2].

\section*{Acknowledgements}
B. H. is supported by NSFC, grant. no. 11401106.

\end{document}